\documentclass{article}
\usepackage{amssymb}
\usepackage{amsfonts}
\usepackage{amsmath}

\newtheorem{theorem}{Theorem}[section]
\newtheorem{acknowledgement}{Acknowledgement}[section]

\newtheorem{conclusion}{Conclusion}[section]

\newtheorem{corollary}{Corollary}[section]

\newtheorem{lemma}{Lemma}[section]

\newtheorem{proposition}{Proposition}[section]

\newenvironment{proof}[1][Proof]{\noindent\textbf{#1.} }{\ \rule{0.5em}{0.5em}}
\numberwithin{equation}{section}
\def\({\left ( }
\def\){\right )}
\def\<{\left < }
\def\>{\right >}
\input{tcilatex}
\begin{document}

\title{On the properties of iterated binomial transforms for the Padovan and
Perrin matrix sequences}
\author{\texttt{Nazmiye Yilmaz} \: and \: \texttt{Necati Taskara}}
\date{Department of Mathematics, Faculty of Science,\\
Selcuk University, Campus, 42075, Konya - Turkey \\
[0.3cm] \textit{nzyilmaz@selcuk.edu.tr} \: and \: \textit{%
ntaskara@selcuk.edu.tr}}
\maketitle

\begin{abstract}
In this study, we apply "$r$" times the binomial transform to the Padovan
and Perrin matrix sequences. Also, the Binet formulas, summations,
generating functions of these transforms are found using recurrence
relations. Finally, we give the relationships of between iterated binomial
transforms for Padovan and Perrin matrix sequences.


\textit{Ams Classification: 11B65, 11B83.}
\end{abstract}

\section{Introduction and Preliminaries}

\bigskip There are so many studies in the literature that concern about the
special number sequences such as Fibonacci, Lucas, Pell, Padovan and Perrin
(see, for example \cite{Koshy,FalconPlaza1,Falcon,ShannonAndersonHoradam},
and the references cited therein). In Fibonacci numbers, there clearly
exists the term Golden ratio which is defined as the ratio of two
consecutive of Fibonacci numbers that converges to $\alpha =\frac{1+\sqrt{5}%
}{2}$. It is also clear that the ratio has so many applications in,
specially, Physics, Engineering, Architecture, etc.\cite{Marek,Marek1}. In a
similar manner, the ratio of two consecutive Padovan and Perrin numbers
converges to 
\begin{equation*}
\alpha _{P}=\sqrt[3]{\dfrac{1}{2}+\dfrac{1}{6}\sqrt{\dfrac{23}{3}}}+\sqrt[3]{%
\dfrac{1}{2}-\dfrac{1}{6}\sqrt{\dfrac{23}{3}}}
\end{equation*}%
that is named as \textit{Plastic constant }and was firstly defined in 1924
by G\'{e}rard Cordonnier. He described applications to architecture and
illustrated the use of the Plastic constant in many buildings.

Although the study of Perrin numbers started in the beginning of 19. century
under different names, the master study was published in 2006 by Shannon et
al. in \cite{ShannonAndersonHoradam}. The authors defined the Perrin $%
\left\{ R_{n}\right\} _{n\in \mathbb{N}}$ and Padovan $\left\{ P_{n}\right\}
_{n\in \mathbb{N}}$ sequences as in the forms 
\begin{equation}
R_{n+3}=R_{n+1}+R_{n},\ \text{where}~\ \ R_{0}=3,\ R_{1}=0,\ R_{2}=2
\label{1.01}
\end{equation}%
and 
\begin{equation}
P_{n+3}=P_{n+1}+P_{n},\ \text{where}~\ \ P_{0}=P_{1}=P_{2}=1\,,  \label{1.02}
\end{equation}%
respectively.

On the other hand, the matrix sequences have taken so much interest for
different type of numbers \cite%
{CivcivTurkmen,GulecTaskara,YazlikTaskaraUsluYilmaz,YilmazTaskara}. For
instance, in \cite{CivcivTurkmen}, authors defined new matrix
generalizations for Fibonacci and Lucas numbers, and using essentially a
matrix approach they showed some properties of these matrix sequences. In 
\cite{GulecTaskara}, Gulec and Taskara gave new generalizations for $(s,t)$%
-Pell and $(s,t)$-Pell Lucas sequences for Pell and Pell--Lucas numbers.
Considering these sequences, they defined the matrix sequences which have
elements of $(s,t)$-Pell and $(s,t)$-Pell Lucas sequences. Also, they
investigated their properties. In \cite{YazlikTaskaraUsluYilmaz}, authors
defined\ a new sequence in which it generalizes $(s,t)$-Fibonacci and\ $%
(s,t) $-Lucas sequences at the same time. After that, by using it, they
established generalized $(s,t)$-matrix sequence. Finally, they presented
some important relationships among this new generalization, $(s,t)$%
-Fibonacci and $(s,t)$-Lucas sequences and their matrix sequences. Moreover,
in \cite{YilmazTaskara}, authors develop the matrix sequences that represent
Padovan and Perrin numbers and examined their properties.

In \cite{YilmazTaskara}, for $n\geqslant 0,$ authors defined Padovan and
Perrin matrix sequences as in the forms 
\begin{equation}
{\mathcal{P}}_{n+3}={\mathcal{P}}_{n+1}+{\mathcal{P}}_{n},  \label{1.03}
\end{equation}%
where%
\begin{equation*}
{\mathcal{P}}_{0}=\left( 
\begin{array}{ccc}
1 & 0 & 0 \\ 
0 & 1 & 0 \\ 
0 & 0 & 1%
\end{array}%
\right) ,{\mathcal{P}}_{1}=\left( 
\begin{array}{ccc}
0 & 1 & 0 \\ 
0 & 0 & 1 \\ 
1 & 1 & 0%
\end{array}%
\right) ,{\mathcal{P}}_{2}=\left( 
\begin{array}{ccc}
0 & 0 & 1 \\ 
1 & 1 & 0 \\ 
0 & 1 & 1%
\end{array}%
\right) ,
\end{equation*}%
and \ 
\begin{equation}
{\mathcal{R}}_{n+3}={\mathcal{R}}_{n+1}+{\mathcal{R}}_{n},  \label{1.04}
\end{equation}%
where 
\begin{equation*}
{\mathcal{R}}_{0}=\left( 
\begin{array}{ccc}
4 & 2 & -3 \\ 
-3 & 1 & 2 \\ 
2 & -1 & 1%
\end{array}%
\right) ,{\mathcal{R}}_{1}=\left( 
\begin{array}{ccc}
-3 & 1 & 2 \\ 
2 & -1 & 1 \\ 
1 & 3 & -1%
\end{array}%
\right) ,{\mathcal{R}}_{2}=\left( 
\begin{array}{ccc}
2 & -1 & 1 \\ 
1 & 3 & -1 \\ 
-1 & 0 & 3%
\end{array}%
\right) .
\end{equation*}

\begin{proposition}
\label{prop1}[10] Let us consider $n,m\geqslant 0,$ the following properties
are hold:
\end{proposition}

\begin{itemize}
\item ${\mathcal{P}}_{m}{\mathcal{P}}_{n}={\mathcal{P}}_{n+m},$

\item ${\mathcal{P}}_{m}{\mathcal{R}}_{n}={\mathcal{R}}_{n}{\mathcal{P}}_{m}=%
{\mathcal{R}}_{n+m}.$
\end{itemize}

In addition, some matrix based transforms can be introduced for a given
sequence. Binomial transform is one of these transforms and there is also
other ones such as rising and falling binomial transforms(see \cite%
{Prodinger,Chen,YazlikYilmazTaskara}). Given an integer sequence $X=\left\{
x_{0},x_{1},x_{2},\ldots \right\} ,$ the binomial transform $B$ of the
sequence $X$, $B\left( X\right) =\left\{ b_{n}\right\} ,$\ is given by%
\begin{equation*}
b_{n}=\sum_{i=0}^{n}\binom{n}{i}x_{i}.
\end{equation*}

In \cite{FalconPlaza,BhadouriaJhalaSingh}, authors gave the application of
the several class of transforms to the $k$-Fibonacci and $k$-Lucas sequence.
In \cite{YilmazTaskara1}, \ the authors applied the binomial transforms to
the Padovan $\left( {\mathcal{P}}_{n}\right) $ and Perrin matrix sequences $%
\left( {\mathcal{R}}_{n}\right) $.

\begin{proposition}
\label{prop2}[14] For $n>0$,
\end{proposition}

\begin{description}
\item[$i)$] Recurrence relation of sequences $\left\{ b_{n}\right\} $ is%
\begin{equation}
b_{n+2}=3b_{n+1}-2b_{n}+b_{n-1},  \label{1.05}
\end{equation}

\item with initial conditions $b_{0}=\left( 
\begin{array}{ccc}
1 & 0 & 0 \\ 
0 & 1 & 0 \\ 
0 & 0 & 1%
\end{array}%
\right) $, $b_{1}=\left( 
\begin{array}{ccc}
1 & 1 & 0 \\ 
0 & 1 & 1 \\ 
1 & 1 & 1%
\end{array}%
\right) $ and $b_{2}=\left( 
\begin{array}{ccc}
1 & 2 & 1 \\ 
1 & 2 & 2 \\ 
2 & 3 & 2%
\end{array}%
\right) .$

\item[$ii)$] Recurrence relation of sequences $\left\{ c_{n}\right\} $ is%
\begin{equation}
c_{n+2}=3c_{n+1}-2c_{n}+c_{n-1},  \label{1.06}
\end{equation}

\item with initial conditions $c_{0}=\left( 
\begin{array}{ccc}
4 & 2 & -3 \\ 
-3 & 1 & 2 \\ 
2 & -1 & 1%
\end{array}%
\right) $, $c_{1}=\left( 
\begin{array}{ccc}
1 & 3 & -1 \\ 
-1 & 0 & 3 \\ 
3 & 2 & 0%
\end{array}%
\right) $ and $c_{2}=\left( 
\begin{array}{ccc}
0 & 3 & 2 \\ 
2 & 2 & 3 \\ 
3 & 5 & 2%
\end{array}%
\right) .$

\item[$iii)$] For $n,m\geq 0,$ we have%
\begin{equation*}
b_{n}b_{m}=b_{n+m},
\end{equation*}
where $n\leq m,$

\item[$iv)$] For $n,m\geq 0,$ we have 
\begin{equation*}
b_{n}c_{m}=c_{n}b_{m}=c_{n+m}.
\end{equation*}
\end{description}

Falcon \cite{Falcon1} studied the iterated application of the some Binomial
transforms to the $k$-Fibonacci sequence. For example, author obtained
recurrence relation of the iterated binomial transform for $k$-Fibonacci
sequence%
\begin{equation*}
a_{k,n+1}^{\left( r\right) }=\left( 2r+k\right) a_{k,n}^{\left( r\right)
}-\left( r^{2}+kr-1\right) a_{k,n-1}^{\left( r\right) },\ \ a_{k,0}^{\left(
r\right) }=0\text{ and }a_{k,1}^{\left( r\right) }=1.
\end{equation*}%
Yilmaz and Taskara \cite{YilmazTaskara3} studied the iterated application of
the some Binomial transforms to the $k$-Lucas sequence. For example, author
obtained recurrence relation of the iterated binomial transform for $k$%
-Lucas sequence%
\begin{equation*}
b_{k,n+1}^{\left( r\right) }=\left( 2r+k\right) b_{k,n}^{\left( r\right)
}-\left( r^{2}+kr-1\right) b_{k,n-1}^{\left( r\right) },\ \ b_{k,0}^{\left(
r\right) }=2\text{ and }b_{k,1}^{\left( r\right) }=2r+k.
\end{equation*}

Motivated by [10,14,17,18], the goal of this paper is to apply iteratly the
binomial transforms to the Padovan $\left( {\mathcal{P}}_{n}\right) $ and
Perrin matrix sequences $\left( {\mathcal{R}}_{n}\right) $. Also, the binet
formulas, summations, generating functions of these transforms are found by
recurrence relations. Finally, it is illustrated the relations between\ of
this transforms by deriving new formulas.

\section{Iterated binomial transforms of the Padovan and Perrin matrix
sequences}

In this section, we will mainly focus on iterated binomial transforms of the
Padovan and Perrin matrix sequences to get some important results. In fact,
we will also present the recurrence relations, binet formulas, summations,
generating functions of these transforms.

The iterated binomial transforms of the Padovan $\left( {\mathcal{P}}%
_{n}\right) $ and Perrin matrix sequences $\left( {\mathcal{R}}_{n}\right) $
are demonstrated by $B^{\left( r\right) }=\left\{ b_{n}^{\left( r\right)
}\right\} $ and $C^{\left( r\right) }=\left\{ c_{n}^{\left( r\right)
}\right\} ,$ where $b_{n}^{\left( r\right) }$ and $c_{n}^{\left( r\right) }$
are obtained by applying $"r"$ times the binomial transform to the Padovan
and Perrin matrix sequences. It are obvious that $b_{0}^{\left( r\right) }={%
\mathcal{P}}_{0}$, $b_{1}^{\left( r\right) }=r{\mathcal{P}}_{0}+{\mathcal{P}}%
_{1},\ b_{2}^{\left( r\right) }=r^{2}{\mathcal{P}}_{0}+2r{\mathcal{P}}_{1}+{%
\mathcal{P}}_{2}\ $ and $c_{0}^{\left( r\right) }=\mathcal{R}_{0}$, $%
c_{1}^{\left( r\right) }=r\mathcal{R}_{0}+\mathcal{R}_{1},\ c_{2}^{\left(
r\right) }=r^{2}\mathcal{R}_{0}+2r\mathcal{R}_{1}+\mathcal{R}_{2}.$

The following lemma will be key of the proof of the next theorems.

\begin{lemma}
\label{lem1}For $n\geq 0$ and $r\geq 1,$ the following equalities are hold:
\end{lemma}

\begin{description}
\item[$i)$] $b_{n+2}^{\left( r\right) }=\sum_{j=0}^{n+1}\binom{n+1}{j}\left(
b_{j}^{\left( r-1\right) }+b_{j+1}^{\left( r-1\right) }\right) ,$

\item[$ii)$] $c_{n+2}^{\left( r\right) }=\sum_{j=0}^{n+1}\binom{n+1}{j}%
\left( c_{j}^{\left( r-1\right) }+c_{j+1}^{\left( r-1\right) }\right) .$
\end{description}

\begin{proof}
Firstly, in here we will just prove $i)$, since $ii)$ can be thought in the
same manner with $i)$.

\begin{description}
\item[$i)$] By using definition of binomial transform, we obtain%
\begin{eqnarray*}
b_{n+2}^{\left( r\right) } &=&\sum_{j=0}^{n+2}\binom{n+2}{j}b_{j}^{\left(
r-1\right) } \\
&=&\sum_{j=1}^{n+2}\binom{n+2}{j}b_{j}^{\left( r-1\right) }+b_{0}^{\left(
r-1\right) }.
\end{eqnarray*}%
And by taking account the well known binomial equality 
\begin{equation*}
\binom{n+2}{i}=\binom{n+1}{i}+\binom{n+1}{i-1},
\end{equation*}%
we have%
\begin{eqnarray*}
b_{n+2}^{\left( r\right) } &=&\sum_{j=1}^{n+1}\binom{n+1}{j}b_{j}^{\left(
r-1\right) }+\sum_{j=1}^{n+2}\binom{n+1}{j-1}b_{j}^{\left( r-1\right)
}+b_{0}^{\left( r-1\right) } \\
&=&\sum_{j=0}^{n+1}\binom{n+1}{j}b_{j}^{\left( r-1\right) }+\sum_{j=0}^{n+1}%
\binom{n+1}{j}b_{j+1}^{\left( r-1\right) } \\
&=&\sum_{j=0}^{n+1}\binom{n+1}{j}\left( b_{j}^{\left( r-1\right)
}+b_{j+1}^{\left( r-1\right) }\right) ,
\end{eqnarray*}%
which is desired result.
\end{description}
\end{proof}

From above Lemma, note that:

\begin{itemize}
\item $b_{n+2}$ is also can be written as $b_{n+2}^{\left( r\right)
}=b_{n+1}^{\left( r\right) }+\sum_{i=0}^{n}\binom{n}{i}\left(
b_{j+1}^{\left( r-1\right) }+b_{j+2}^{\left( r-1\right) }\right) $,

\item $c_{n+2}$ is also can be written as $c_{n+2}^{\left( r\right)
}=c_{n+1}^{\left( r\right) }+\sum_{i=0}^{n}\binom{n}{i}\left(
c_{j+1}^{\left( r-1\right) }+c_{j+2}^{\left( r-1\right) }\right) $.
\end{itemize}

\begin{theorem}
\label{teo1}For $n\geq 0$ and $r\geq 1,$ the recurrence relations of
sequences $\left\{ b_{n}^{\left( r\right) }\right\} $ and $\left\{
c_{n}^{\left( r\right) }\right\} $ are%
\begin{equation}
b_{n+2}^{\left( r\right) }=3rb_{n+1}^{\left( r\right) }-\left(
3r^{2}-1\right) b_{n}^{\left( r\right) }+\left( r^{3}-r+1\right)
b_{n-1}^{\left( r\right) },  \label{2.1}
\end{equation}%
\begin{equation}
c_{n+2}^{\left( r\right) }=3rc_{n+1}^{\left( r\right) }-\left(
3r^{2}-1\right) c_{n}^{\left( r\right) }+\left( r^{3}-r+1\right)
c_{n-1}^{\left( r\right) },  \label{2.2}
\end{equation}%
with initial conditions $b_{0}^{\left( r\right) }={\mathcal{P}}_{0}$, $%
b_{1}^{\left( r\right) }=r{\mathcal{P}}_{0}+{\mathcal{P}}_{1},\
b_{2}^{\left( r\right) }=r^{2}{\mathcal{P}}_{0}+2r{\mathcal{P}}_{1}+{%
\mathcal{P}}_{2}\ $ and $c_{0}^{\left( r\right) }=\mathcal{R}_{0}$, $%
c_{1}^{\left( r\right) }=r\mathcal{R}_{0}+\mathcal{R}_{1},\ c_{2}^{\left(
r\right) }=r^{2}\mathcal{R}_{0}+2r\mathcal{R}_{1}+\mathcal{R}_{2}.$
\end{theorem}

\begin{proof}
Similarly the proof of the Lemma 2.1, only the first case, the equation (\ref%
{2.1}) will be proved. We will omit the equation (\ref{2.2}) since the
proofs will not be different.

The proof will be done by induction steps on $r$ and $n$.

First of all, for $r=1,$ from the $i)$ condition of Proposition 1.2, it is
true $b_{n+2}=3b_{n+1}-2b_{n}+b_{n-1}.$

Let us consider definition of iterated binomial transform, then we have 
\begin{equation*}
b_{3}^{\left( r\right) }=\left( r^{3}+1\right) {\mathcal{P}}_{0}+\left(
3r^{2}+1\right) {\mathcal{P}}_{1}+3r{\mathcal{P}}_{2}\ .
\end{equation*}%
The initial conditions are%
\begin{equation*}
b_{0}^{\left( r\right) }={\mathcal{P}}_{0},b_{1}^{\left( r\right) }=r{%
\mathcal{P}}_{0}+{\mathcal{P}}_{1},\ b_{2}^{\left( r\right) }=r^{2}{\mathcal{%
P}}_{0}+2r{\mathcal{P}}_{1}+{\mathcal{P}}_{2}.
\end{equation*}%
Hence, for $n=1,$ the equality (\ref{2.1}) is true, that is $b_{3}^{\left(
r\right) }=3rb_{2}^{\left( r\right) }-\left( 3r^{2}-1\right) b_{1}^{\left(
r\right) }+\left( r^{3}-r+1\right) b_{0}^{\left( r\right) }.$

Actually, by assuming the equation in (\ref{2.1}) holds for all $(r-1,n)$
and $(r,n-1),$ that is,%
\begin{equation*}
b_{n+2}^{\left( r-1\right) }=3\left( r-1\right) b_{n+1}^{\left( r-1\right)
}-\left( 3\left( r-1\right) ^{2}-1\right) b_{n}^{\left( r-1\right) }+\left(
\left( r-1\right) ^{3}-\left( r-1\right) +1\right) b_{n-1}^{\left(
r-1\right) },
\end{equation*}%
and%
\begin{equation*}
b_{n+1}^{\left( r\right) }=3rb_{n}^{\left( r\right) }-\left( 3r^{2}-1\right)
b_{n-1}^{\left( r\right) }+\left( r^{3}-r+1\right) b_{n-2}^{\left( r\right)
}.
\end{equation*}%
Then, we need to show that it is true for $\left( r,n\right) .$ That is, 
\begin{equation*}
b_{n+2}^{\left( r\right) }=3rb_{n+1}^{\left( r\right) }-\left(
3r^{2}-1\right) b_{n}^{\left( r\right) }+\left( r^{3}-r+1\right)
b_{n-1}^{\left( r\right) }.
\end{equation*}%
Let us label $b_{n+2}^{\left( r\right) }=3rb_{n+1}^{\left( r\right) }-\left(
3r^{2}-1\right) b_{n}^{\left( r\right) }+\left( r^{3}-r+1\right)
b_{n-1}^{\left( r\right) }$ by $RHS$. Hence, we can write%
\begin{eqnarray*}
RHS &=&3r\sum_{j=0}^{n+1}\binom{n+1}{j}b_{j}^{\left( r-1\right) }-\left(
3r^{2}-1\right) \sum_{j=0}^{n}\binom{n}{j}b_{j}^{\left( r-1\right) }+\left(
r^{3}-r+1\right) \sum_{j=0}^{n-1}\binom{n-1}{j}b_{j}^{\left( r-1\right) } \\
&=&\left( 3r-3r^{2}+1\right) \sum_{j=0}^{n+1}\binom{n+1}{j}b_{j}^{\left(
r-1\right) }+\left( 3r^{2}-1\right) \sum_{j=0}^{n+1}\binom{n+1}{j}%
b_{j}^{\left( r-1\right) } \\
&&-\left( 3r^{2}-1\right) \sum_{j=0}^{n}\binom{n}{j}b_{j}^{\left( r-1\right)
}+\left( r^{3}-r+1\right) \sum_{j=0}^{n-1}\binom{n-1}{j}b_{j}^{\left(
r-1\right) } \\
&=&\left( 3r-3r^{2}+1\right) \sum_{j=0}^{n+1}\binom{n+1}{j}b_{j}^{\left(
r-1\right) }+\left( 3r^{2}-1\right) \sum_{j=1}^{n+1}\binom{n}{j-1}%
b_{j}^{\left( r-1\right) } \\
&&+\left( r^{3}-r+1\right) \sum_{j=0}^{n-1}\binom{n-1}{j}b_{j}^{\left(
r-1\right) } \\
&=&\left( 3r-3r^{2}+1\right) \sum_{j=0}^{n+1}\binom{n+1}{j}b_{j}^{\left(
r-1\right) }+\left( 3r-3r^{2}+1\right) \sum_{j=1}^{n+1}\binom{n}{j-1}%
b_{j}^{\left( r-1\right) } \\
&&+\left( 6r^{2}-3r-2\right) \sum_{j=1}^{n+1}\binom{n}{j-1}b_{j}^{\left(
r-1\right) }+\left( r^{3}-r+1\right) \sum_{j=0}^{n-1}\binom{n-1}{j}%
b_{j}^{\left( r-1\right) } \\
&=&\left( 3r-3r^{2}+1\right) \sum_{j=0}^{n+1}\binom{n+1}{j}b_{j}^{\left(
r-1\right) }+\left( 3r-3r^{2}+1\right) \sum_{j=0}^{n+1}\binom{n+1}{j}%
b_{j+1}^{\left( r-1\right) } \\
&&+\left( 6r^{2}-3r-2\right) \sum_{j=1}^{n+1}\binom{n}{j-1}b_{j}^{\left(
r-1\right) }+\left( r^{3}-r+1\right) \sum_{j=0}^{n-1}\binom{n-1}{j}%
b_{j}^{\left( r-1\right) } \\
&&-\left( 3r-3r^{2}+1\right) \sum_{j=1}^{n+1}\binom{n}{j-1}b_{j+1}^{\left(
r-1\right) }.
\end{eqnarray*}%
From Lemma 2.1, we have%
\begin{eqnarray*}
RHS &=&\left( 3r-3r^{2}+1\right) b_{n+2}^{\left( r\right) }+\left(
6r^{2}-3r-2\right) \sum_{j=0}^{n}\binom{n}{j}b_{j+1}^{\left( r-1\right) } \\
&&+\left( r^{3}-r+1\right) \sum_{j=0}^{n-1}\binom{n-1}{j}b_{j}^{\left(
r-1\right) }+\left( 3r^{2}-3r-1\right) \sum_{j=0}^{n}\binom{n}{j}%
b_{j+2}^{\left( r-1\right) } \\
&=&\left( 3r-3r^{2}+1\right) b_{n+2}^{\left( r\right) }+\left(
r^{3}-r+1\right) \sum_{j=0}^{n-1}\binom{n-1}{j}b_{j}^{\left( r-1\right) } \\
&&+\left( 3r^{2}-3r\right) \sum_{j=0}^{n}\binom{n}{j}\left( b_{j+1}^{\left(
r-1\right) }+b_{j+2}^{\left( r-1\right) }\right) +\left( 3r^{2}-2\right)
\sum_{j=0}^{n}\binom{n}{j}b_{j+1}^{\left( r-1\right) } \\
&&-\sum_{j=0}^{n}\binom{n}{j}b_{j+2}^{\left( r-1\right) }
\end{eqnarray*}%
\begin{eqnarray*}
RHS &=&b_{n+2}^{\left( r\right) }+\left( 3r-3r^{2}\right) b_{n+1}^{\left(
r\right) }+\left( r^{3}-r+1\right) \sum_{j=0}^{n-1}\binom{n-1}{j}%
b_{j}^{\left( r-1\right) } \\
&&+\left( 3r^{2}-2\right) \sum_{j=1}^{n}\binom{n-1}{j-1}b_{j+1}^{\left(
r-1\right) }+\left( 3r^{2}-2\right) \sum_{j=0}^{n-1}\binom{n-1}{j}%
b_{j+1}^{\left( r-1\right) } \\
&&-\sum_{j=0}^{n}\binom{n}{j}b_{j+2}^{\left( r-1\right) } \\
&=&b_{n+2}^{\left( r\right) }+\left( 3r-3r^{2}\right) b_{n+1}^{\left(
r\right) }+\left( 3r^{2}-3r\right) b_{n}^{\left( r\right) } \\
&&+\left( r^{3}-3r^{2}+2r+1\right) \sum_{j=0}^{n-1}\binom{n-1}{j}%
b_{j}^{\left( r-1\right) }+\left( 3r-2\right) \sum_{j=0}^{n-1}\binom{n-1}{j}%
b_{j+1}^{\left( r-1\right) } \\
&&+\left( 3r^{2}-2\right) \sum_{j=0}^{n-1}\binom{n-1}{j}b_{j+2}^{\left(
r-1\right) }-\sum_{j=0}^{n}\binom{n}{j}b_{j+2}^{\left( r-1\right) } \\
&=&b_{n+2}^{\left( r\right) }+\left( 3r-3r^{2}\right) \left( b_{n+1}^{\left(
r\right) }-b_{n}^{\left( r\right) }\right) \\
&&+\sum_{j=0}^{n-1}\binom{n-1}{j}\left[ 
\begin{array}{c}
\left( r^{3}-3r^{2}+2r+1\right) b_{j}^{\left( r-1\right) }+\left(
3r-2\right) b_{j+1}^{\left( r-1\right) } \\ 
++\left( 3r^{2}-3\right) b_{j+2}^{\left( r-1\right) }-b_{j+3}^{\left(
r-1\right) }%
\end{array}%
\right] .
\end{eqnarray*}%
Afterward, by taking account assumption and Lemma 2.1, we deduce%
\begin{eqnarray*}
RHS &=&b_{n+2}^{\left( r\right) }+\left( 3r-3r^{2}\right) \left(
b_{n+1}^{\left( r\right) }-b_{n}^{\left( r\right) }\right) \\
&&-\left( 3r-3r^{2}\right) \sum_{j=0}^{n-1}\binom{n-1}{j}\left(
b_{j+1}^{\left( r-1\right) }+b_{j+2}^{\left( r-1\right) }\right) \\
&=&b_{n+2}^{\left( r\right) }
\end{eqnarray*}%
which is completed the proof of this theorem.
\end{proof}

The characteristic equation of sequences $\left\{ b_{n}^{\left( r\right)
}\right\} $ and $\left\{ c_{n}^{\left( r\right) }\right\} $ in (\ref{2.1})
and (\ref{2.2}) is \linebreak $\lambda ^{3}-3r\lambda ^{2}+\left(
3r^{2}-1\right) \lambda -\left( r^{3}-r+1\right) =0.$ Let be $\lambda _{1},$ 
$\lambda _{2}$ and $\lambda _{3}$ the roots of this equation. Then, binet's
formulas of sequences $\left\{ b_{n}^{\left( r\right) }\right\} $ and $%
\left\{ c_{n}^{\left( r\right) }\right\} $ can be expressed as%
\begin{equation}
b_{n}^{\left( r\right) }=X_{1}\lambda _{1}^{n}+Y_{1}\lambda
_{2}^{n}+Z_{1}\lambda _{3}^{n},  \label{2.3}
\end{equation}%
and%
\begin{equation}
c_{n}^{\left( r\right) }=X_{2}\lambda _{1}^{n}+Y_{2}\lambda
_{2}^{n}+Z_{2}\lambda _{3}^{n},  \label{2.4}
\end{equation}%
where%
\begin{eqnarray*}
X_{1} &=&\frac{\lambda _{1}b_{2}^{\left( r\right) }-\left( 3r\lambda
_{1}-\lambda _{1}^{2}\right) b_{1}^{\left( r\right) }+\left(
r^{3}-r+1\right) b_{0}^{\left( r\right) }}{\lambda _{1}\left( \lambda
_{1}-\lambda _{2}\right) \left( \lambda _{1}-\lambda _{3}\right) }, \\
Y_{1} &=&\frac{\lambda _{2}b_{2}^{\left( r\right) }-\left( 3r\lambda
_{2}-\lambda _{2}^{2}\right) b_{1}^{\left( r\right) }+\left(
r^{3}-r+1\right) b_{0}^{\left( r\right) }}{\lambda _{2}\left( \lambda
_{2}-\lambda _{1}\right) \left( \lambda _{2}-\lambda _{3}\right) }, \\
Z_{1} &=&\frac{\lambda _{3}b_{2}^{\left( r\right) }-\left( 3r\lambda
_{3}-\lambda _{3}^{2}\right) b_{1}^{\left( r\right) }+\left(
r^{3}-r+1\right) b_{0}^{\left( r\right) }}{\lambda _{3}\left( \lambda
_{3}-\lambda _{1}\right) \left( \lambda _{3}-\lambda _{2}\right) },
\end{eqnarray*}%
and 
\begin{eqnarray*}
X_{2} &=&\frac{\lambda _{1}c_{2}^{\left( r\right) }-\left( 3r\lambda
_{1}-\lambda _{1}^{2}\right) c_{1}^{\left( r\right) }+\left(
r^{3}-r+1\right) c_{0}^{\left( r\right) }}{\lambda _{1}\left( \lambda
_{1}-\lambda _{2}\right) \left( \lambda _{1}-\lambda _{3}\right) }, \\
Y_{2} &=&\frac{\lambda _{2}c_{2}^{\left( r\right) }-\left( 3r\lambda
_{2}-\lambda _{2}^{2}\right) c_{1}^{\left( r\right) }+\left(
r^{3}-r+1\right) c_{0}^{\left( r\right) }}{\lambda _{2}\left( \lambda
_{2}-\lambda _{1}\right) \left( \lambda _{2}-\lambda _{3}\right) }, \\
Z_{2} &=&\frac{\lambda _{3}c_{2}^{\left( r\right) }-\left( 3r\lambda
_{3}-\lambda _{3}^{2}\right) c_{1}^{\left( r\right) }+\left(
r^{3}-r+1\right) c_{0}^{\left( r\right) }}{\lambda _{3}\left( \lambda
_{3}-\lambda _{1}\right) \left( \lambda _{3}-\lambda _{2}\right) }.
\end{eqnarray*}%
\qquad \qquad

Now, we give the sums of iterated binomial transforms for the Padovan and
Perrin matrix sequences.

\begin{theorem}
Sums of $\ $sequences $\left\{ b_{n}^{\left( r\right) }\right\} $ and $%
\left\{ c_{n}^{\left( r\right) }\right\} $ are

\begin{description}
\item[$i)$] $\sum_{i=0}^{n-1}b_{i}^{\left( r\right) }=\dfrac{b_{n+1}^{\left(
r\right) }+\left( 1-3r\right) b_{n}^{\left( r\right) }+\left(
r^{3}-r+1\right) b_{n-1}^{\left( r\right) }+\left( 3r-1\right) b_{0}^{\left(
r\right) }-b_{1}^{\left( r\right) }}{r^{3}+2r}$

\item[$ii)$] $\sum_{i=0}^{n-1}c_{i}^{\left( r\right) }=\dfrac{%
c_{n+1}^{\left( r\right) }+\left( 1-3r\right) c_{n}^{\left( r\right)
}+\left( r^{3}-r+1\right) c_{n-1}^{\left( r\right) }+\left( 3r-1\right)
c_{0}^{\left( r\right) }-c_{1}^{\left( r\right) }}{r^{3}+2r}.$
\end{description}
\end{theorem}

\begin{proof}
We omit Padovan case since the proof be quite similar. By considering
equation (\ref{2.4}), we have%
\begin{equation*}
\sum\limits_{i=0}^{n-1}c_{i}^{\left( r\right)
}=\sum\limits_{i=0}^{n-1}\left( X_{2}\lambda _{1}^{n}+Y_{2}\lambda
_{2}^{n}+Z_{2}\lambda _{3}^{n}\right) .
\end{equation*}%
Then we obtain%
\begin{equation*}
\sum\limits_{i=0}^{n-1}c_{i}^{\left( r\right) }=X_{2}\left( \frac{\lambda
_{1}^{n}-1}{\lambda _{1}-1}\right) +Y_{2}\left( \frac{\lambda _{2}^{n}-1}{%
\lambda _{2}-1}\right) +Z_{2}\left( \frac{\lambda _{3}^{n}-1}{\lambda _{3}-1}%
\right) .
\end{equation*}%
Afterward, by taking account equations $c_{-1}^{\left( r\right) }=0,\
\lambda _{1}\cdot \lambda _{2}\cdot \lambda _{3}=r^{3}-r+1$ and $\lambda
_{1}+\lambda _{2}+\lambda _{3}=3r,$\ we conclude 
\begin{equation*}
\sum\limits_{i=0}^{n-1}c_{i}^{\left( r\right) }=\dfrac{c_{n+1}^{\left(
r\right) }+\left( 1-3r\right) c_{n}^{\left( r\right) }+\left(
r^{3}-r+1\right) c_{n-1}^{\left( r\right) }+\left( 3r-1\right) c_{0}^{\left(
r\right) }-c_{1}^{\left( r\right) }}{r^{3}+2r}.
\end{equation*}
\end{proof}

\begin{theorem}
\label{teo2}The generating functions of the iterated binomial transforms for 
$\left( {\mathcal{P}}_{n}\right) $ and $\left( {\mathcal{R}}_{n}\right) $\
are

\begin{description}
\item[$i)$] $\sum\limits_{i=0}^{\infty }b_{i}^{\left( r\right) }x^{i}=\dfrac{%
b_{0}^{\left( r\right) }+\left( b_{1}^{\left( r\right) }-3rb_{0}^{\left(
r\right) }\right) x+\left( b_{2}^{\left( r\right) }-3rb_{1}^{\left( r\right)
}+\left( 3r^{2}-1\right) b_{0}^{\left( r\right) }\right) x^{2}}{1-3rx+\left(
3r^{2}-1\right) x^{2}-\left( r^{3}-r+1\right) x^{3}},$

\item[$ii)$] $\sum\limits_{i=0}^{\infty }c_{i}^{\left( r\right) }x^{i}=%
\dfrac{c_{0}^{\left( r\right) }+\left( c_{1}^{\left( r\right)
}-3rc_{0}^{\left( r\right) }\right) x+\left( c_{2}^{\left( r\right)
}-3rc_{1}^{\left( r\right) }+\left( 3r^{2}-1\right) c_{0}^{\left( r\right)
}\right) x^{2}}{1-3rx+\left( 3r^{2}-1\right) x^{2}-\left( r^{3}-r+1\right)
x^{3}}.$
\end{description}
\end{theorem}

\begin{proof}

\begin{description}
\item[$i)$] Assume that $b\left( x,r\right) =\sum\limits_{i=0}^{\infty
}b_{i}^{\left( r\right) }x^{i}$ is the generating function of the iterated
binomial transform for $\left( {\mathcal{P}}_{n}\right) $. From Theorem \ref%
{teo1}, we obtain 
\begin{eqnarray*}
b\left( x,r\right) &=&b_{0}^{\left( r\right) }+b_{1}^{\left( r\right)
}x+b_{2}^{\left( r\right) }x^{2} \\
&&+\sum\limits_{i=3}^{\infty }\left( 3rb_{i-1}^{\left( r\right) }-\left(
3r^{2}-1\right) b_{i-2}^{\left( r\right) }+\left( r^{3}-r+1\right)
b_{i-3}^{\left( r\right) }\right) x^{i} \\
&=&b_{0}^{\left( r\right) }+b_{1}^{\left( r\right) }x+b_{2}^{\left( r\right)
}x^{2}+3rx\sum\limits_{i=3}^{\infty }b_{i-1}^{\left( r\right) }x^{i-1} \\
&&-\left( 3r^{2}-1\right) x^{2}\sum\limits_{i=3}^{\infty }b_{i-2}^{\left(
r\right) }x^{i-2}+\left( r^{3}-r+1\right) x^{3}\sum\limits_{i=3}^{\infty
}b_{i-3}^{\left( r\right) }x^{i-3} \\
&=&b_{0}^{\left( r\right) }+b_{1}^{\left( r\right) }x+b_{2}^{\left( r\right)
}x^{2}+3rx\sum\limits_{i=0}^{\infty }b_{i}^{\left( r\right) }x^{i}-3rx\left(
b_{0}^{\left( r\right) }+b_{1}^{\left( r\right) }x\right) \\
&&-\left( 3r^{2}-1\right) x^{2}\sum\limits_{i=0}^{\infty }b_{i}^{\left(
r\right) }x^{i}+\left( 3r^{2}-1\right) x^{2}b_{0}^{\left( r\right) }+\left(
r^{3}-r+1\right) x^{3}\sum\limits_{i=0}^{\infty }b_{i}^{\left( r\right)
}x^{i}.
\end{eqnarray*}%
Now rearrangement the equation implies that 
\begin{equation*}
b\left( x,r\right) =\dfrac{b_{0}^{\left( r\right) }+\left( b_{1}^{\left(
r\right) }-3rb_{0}^{\left( r\right) }\right) x+\left( b_{2}^{\left( r\right)
}-3rb_{1}^{\left( r\right) }+\left( 3r^{2}-1\right) b_{0}^{\left( r\right)
}\right) x^{2}}{1-3rx+\left( 3r^{2}-1\right) x^{2}-\left( r^{3}-r+1\right)
x^{3}},
\end{equation*}%
which equal to the $\sum\limits_{i=0}^{\infty }b_{i}^{\left( r\right) }x^{i}$
in theorem. Hence the result.

\item[$ii)$] The proof of generating function of the iterated binomial
transform for Perrin matrix sequences can see by taking account proof of $i)$%
.
\end{description}
\end{proof}

\section{The relationships between new iterated binomial transforms}

In this section, we present the relationships between the iterated binomial
transform of the Padovan matrix sequence and iterated binomial transform of
the\ Perrin matrix sequence.

\begin{theorem}
\label{teo3}For $n,m\geq 0,$ we have
\end{theorem}

\begin{description}
\item[$i)$] $b_{n}^{\left( r\right) }b_{m}^{\left( r\right)
}=b_{n+m}^{\left( r\right) },$ where $n\leq m,$

\item[$ii)$] $b_{n}^{\left( r\right) }c_{m}^{\left( r\right) }=c_{m}^{\left(
r\right) }b_{n}^{\left( r\right) }=c_{n+m}^{\left( r\right) },$
\end{description}

\begin{proof}

\begin{description}
\item[$i)$] The proof will be done by induction step on $r$. First of all,
for $r=1,$ from the $iii)$ condition of Proposition 1.2, it is true $%
b_{n}b_{m}=b_{n+m}.$

Actually, by assuming the equation in $i)$ holds for all $r,$ that is,%
\begin{eqnarray*}
b_{n}^{\left( r\right) }b_{m}^{\left( r\right) } &=&\sum\limits_{i=0}^{n}%
\binom{n}{i}b_{i}^{\left( r-1\right) }+\sum\limits_{j=0}^{m}\binom{m}{j}%
b_{j}^{\left( r-1\right) } \\
&=&\sum\limits_{k=0}^{n+m}\binom{n+m}{k}b_{k}^{\left( r-1\right)
}=b_{n+m}^{\left( r\right) }.
\end{eqnarray*}%
Then, we need to show that it is true for $r+1.$ That is, From definition of
iterated binomial transform, we have 
\begin{eqnarray*}
b_{n}^{\left( r+1\right) }b_{m}^{\left( r+1\right) } &=&\left(
\sum\limits_{i=0}^{n}\binom{n}{i}b_{i}^{\left( r\right) }\right) \left(
\sum\limits_{j=0}^{m}\binom{m}{j}b_{j}^{\left( r\right) }\right) \\
&=&\left( \sum\limits_{i=0}^{n}\binom{n}{i}\sum\limits_{k=0}^{i}\binom{i}{k}%
b_{k}^{\left( r-1\right) }\right) \left( \sum\limits_{j=0}^{m}\binom{m}{j}%
\sum\limits_{l=0}^{j}\binom{j}{l}b_{l}^{\left( r-1\right) }\right) .
\end{eqnarray*}%
And, by considering assumption, we obtain 
\begin{eqnarray*}
b_{n}^{\left( r+1\right) }b_{m}^{\left( r+1\right) }
&=&\sum\limits_{i=0}^{n}\sum\limits_{j=0}^{m}\binom{n}{i}\binom{m}{j}%
b_{i+j}^{\left( r\right) } \\
&=&\binom{n}{0}\binom{m}{0}b_{0}^{\left( r\right) }+\binom{n}{0}\binom{m}{1}%
b_{1}^{\left( r\right) }+\cdots +\binom{n}{0}\binom{m}{m}b_{m}^{\left(
r\right) } \\
&&+\binom{n}{1}\binom{m}{0}b_{1}^{\left( r\right) }+\binom{n}{1}\binom{m}{1}%
b_{2}^{\left( r\right) }+\cdots +\binom{n}{1}\binom{m}{m}b_{m+1}^{\left(
r\right) }+ \\
&&\vdots \\
&&+\binom{n}{n}\binom{m}{0}b_{n}^{\left( r\right) }+\binom{n}{n}\binom{m}{1}%
b_{n+1}^{\left( r\right) }+\cdots +\binom{n}{n}\binom{m}{m}b_{n+m}^{\left(
r\right) } \\
&=&\binom{n}{0}\binom{m}{0}b_{0}^{\left( r\right) }+\left[ \binom{n}{0}%
\binom{m}{1}+\binom{n}{1}\binom{m}{0}\right] b_{1}^{\left( r\right) } \\
&&+\left[ \binom{n}{0}\binom{m}{2}+\binom{n}{1}\binom{m}{1}+\binom{n}{2}%
\binom{m}{0}\right] b_{2}^{\left( r\right) }+\cdots \\
&&+\left[ \binom{n}{0}\binom{m}{k}+\binom{n}{1}\binom{m}{k-1}+\cdots +\binom{%
n}{k}\binom{m}{0}\right] b_{k}^{\left( r\right) }+\cdots \\
&&+\binom{n}{n}\binom{m}{m}b_{n+m}^{\left( r\right) }.
\end{eqnarray*}%
By taking account Vandermonde identity $\sum\limits_{j=0}^{k}\binom{x}{j}%
\binom{y}{k-j}=\binom{x+y}{k},$ we get%
\begin{eqnarray*}
b_{n}^{\left( r+1\right) }b_{m}^{\left( r+1\right) } &=&\binom{n+m}{0}%
b_{0}^{\left( r\right) }+\binom{n+m}{1}b_{1}^{\left( r\right) }+\binom{n+m}{2%
}b_{2}^{\left( r\right) }+\cdots \\
&&+\binom{n+m}{k}b_{k}^{\left( r\right) }+\cdots +\binom{n+m}{n+m}%
b_{n+m}^{\left( r\right) } \\
&=&\sum\limits_{i=0}^{n+m}\binom{n+m}{i}b_{i}^{\left( r\right) } \\
&=&b_{n+m}^{\left( r+1\right) }.
\end{eqnarray*}

\item[$ii)$] The proof is similar proof of $i)$.
\end{description}
\end{proof}

\begin{theorem}
\label{teo4}The properties of the transforms $\left\{ b_{n}^{\left( r\right)
}\right\} $ and $\left\{ c_{n}^{\left( r\right) }\right\} $ would be
illustrated by following way:
\end{theorem}

\begin{description}
\item[$i)$] $b_{n+1}^{\left( r\right) }-b_{n}^{\left( r\right)
}=b_{1}^{\left( r-1\right) }b_{n}^{\left( r\right) },$

\item[$ii)$] $c_{n+1}^{\left( r\right) }-c_{n}^{\left( r\right)
}=b_{1}^{\left( r-1\right) }c_{n}^{\left( r\right) },$

\item[$iii)$] $c_{n+1}^{\left( r\right) }-c_{n}^{\left( r\right)
}=c_{1}^{\left( r-1\right) }b_{n}^{\left( r\right) }.$
\end{description}

\begin{proof}
We will omit the proof of $ii)$ and $iii)$, since it is quite similar with $%
i).$ Therefore, by considering definition of iterated binomial transform and
Lemma \ref{lem1}-$i),$ we have%
\begin{equation*}
b_{n+1}^{\left( r\right) }-b_{n}^{\left( r\right) }=\sum\limits_{i=0}^{n}%
\binom{n}{i}b_{i+1}^{\left( r-1\right) }.
\end{equation*}%
From Theorem \ref{teo3}-$i),$ we get 
\begin{equation*}
b_{n+1}^{\left( r\right) }-b_{n}^{\left( r\right) }=\sum\limits_{i=0}^{n}%
\binom{n}{i}b_{i}^{\left( r-1\right) }b_{1}^{\left( r-1\right)
}=b_{1}^{\left( r-1\right) }b_{n}^{\left( r\right) }.
\end{equation*}
\end{proof}

\begin{theorem}
\label{teo5}For $n,m\geq 0,$ the relation between the transforms $\left\{
b_{n}^{\left( r\right) }\right\} $ and $\left\{ c_{n}^{\left( r\right)
}\right\} $ is%
\begin{equation*}
c_{m}^{\left( r-1\right) }b_{n}^{\left( r\right) }=b_{m}^{\left( r-1\right)
}c_{n}^{\left( r\right) }.
\end{equation*}
\end{theorem}

\begin{proof}
By considering definition of iterated binomial transform, we have%
\begin{eqnarray*}
c_{m}^{\left( r-1\right) }b_{n}^{\left( r\right) } &=&c_{m}^{\left(
r-1\right) }\sum\limits_{i=0}^{n}\binom{n}{i}b_{i}^{\left( r-1\right) } \\
&=&\sum\limits_{i=0}^{n}\binom{n}{i}c_{m}^{\left( r-1\right) }b_{i}^{\left(
r-1\right) }.
\end{eqnarray*}%
From Theorem \ref{teo3}-$ii),$ we get%
\begin{eqnarray*}
c_{m}^{\left( r-1\right) }b_{n}^{\left( r\right) } &=&\sum\limits_{i=0}^{n}%
\binom{n}{i}c_{m+i}^{\left( r-1\right) } \\
&=&\sum\limits_{i=0}^{n}\binom{n}{i}b_{m}^{\left( r-1\right) }c_{i}^{\left(
r-1\right) } \\
&=&b_{m}^{\left( r-1\right) }c_{n}^{\left( r\right) }.
\end{eqnarray*}
\end{proof}

By choosing $m=0$ in Theorem \ref{teo5} and using the initial conditions of
equations (\ref{2.1}) and (\ref{2.2}), we obtain the following corollary.

\begin{corollary}
The following equalities are hold:
\end{corollary}

\begin{description}
\item[$i)$] $c_{n}^{\left( r\right) }=\mathcal{R}_{0}b_{n}^{\left( r\right)
},$

\item[$ii)$] $b_{n}^{\left( r\right) }=\mathcal{R}_{0}^{-1}c_{n}^{\left(
r\right) }.$
\end{description}

\begin{corollary}
We should note that choosing $r=1$ in the all results of Section 2 and 3, it
is actually obtained some properties of the iterated binomial transforms for
Padovan and Perrin matrix sequences such that the recurrence relations,
Binet formulas, summations, generating functions and relationships of
between binomial transforms for Padovan and Perrin matrix sequences.
\end{corollary}

\begin{conclusion}
In this paper, we define the iterated binomial transforms for Padovan and
Perrin matrix sequences and present some properties of these transforms. By
the results in Sections 2 and 3 of this paper, we have a great opportunity
to compare and obtain some new properties over these transforms. This is the
main aim of this paper. Thus, we extend some recent result in the literature.

In the future studies on the iterated binomial transform for number
sequences, we except that the following topics will bring a new insight.

\begin{description}
\item[(1)] It would be interesting to study the iterated binomial transform
for Fibonacci and Lucas matrix sequences,

\item[(2)] Also, it would be interesting to study the iterated binomial
transform for Pell and Pell-Lucas matrix sequences.
\end{description}
\end{conclusion}

\begin{acknowledgement}
This research is supported by TUBITAK and Selcuk University Scientific
Research Project Coordinatorship (BAP).
\end{acknowledgement}

\end{document}